\documentclass[11pt]{article}
\usepackage{amssymb}
\usepackage{amsmath}
\usepackage{amsfonts}
\usepackage{graphicx}
\usepackage{subfig}

\numberwithin{equation}{section}
\newtheorem{theorem}{Theorem}[section]

\newtheorem{claim}[theorem]{Claim}

\newtheorem{definition}[theorem]{Definition}

\newtheorem{lemma}[theorem]{Lemma}

\newtheorem{proposition}[theorem]{Proposition}

\newenvironment{proof}[1][Proof]{\noindent\textbf{#1.} }{\ \rule{0.5em}{0.5em}}

\begin{document}

\title{{\huge \textbf{Characterizations of Normal and Binormal Surfaces in $%
G_{3}$ }}}
\author{Dae Won Yoon $^\dag$ and Z\"{u}hal K\"{u}\c{c}\"{u}karslan Y\"{u}zba%
\c{s}\i\ $^\ddag$ }
\date{}
\maketitle
{\footnotesize {%
\centerline  {$\dag$ Department of Mathematics  Education
and RINS }}}

{\footnotesize \centerline  {Gyeongsang National University }}

{\footnotesize \centerline{ Jinju 52828, Republic of Korea} }

{\footnotesize \centerline { {E-mail address}:   {\tt
dwyoon@gnu.ac.kr}} }

{\footnotesize \vskip 0.2 cm {%
\centerline  { $\ddag$ Department of
Mathematics}}}

{\footnotesize {\centerline  {F\i rat University }}}

{\footnotesize \centerline{23119 Elazig, Turkey } }

{\footnotesize 
\centerline { {E-mail address}:   {\tt zuhal2387@yahoo.com.tr
}} }
\begin{abstract}
In this paper, our aim is to give surfaces in the Galilean 3-space $\mathbf{G%
}_{3}$ with the property that there exist four geodesics through each point
such that every surface built with the normal lines and the binormal lines
along these geodesics is a surface with a minimal surface and a constant
negative Gaussian curvature. We show that $\psi $ should be an isoparametric
surface in $\mathbf{G}_{3}$: A plane or a circular hyperboloid.
\end{abstract}

\renewcommand\leftmark {\centerline{  \rm Isoparametric Surfaces Via Normal Surfaces in $G_{3}$ }}
\renewcommand\rightmark {\centerline{ \rm  Isoparametric Surfaces Via Normal Surfaces in $G_{3}$  }}

\renewcommand{\thefootnote}{} \footnote{{\ {Corresponding author:}} Zuhal K.
Yuzbasi.} \footnote{%
2010 \textit{AMS Mathematics Subject Classification:} 53A35, 53Z05.}

\footnote{{\ {Key words and phrases: }} Normal surface, Binormal surface,
Geodesics, Galilean space.
\par
The first author was supported by Basic Science Research Program through the
National Research Foundation of Korea funded by the Ministry of
Education(2015R1D1A1A01060046)}

\renewcommand{\thefootnote}{\arabic{footnote}} \setcounter{footnote}{0}

\section{Introduction}

A helical curve (or a helix) is a geometric curve which has non-vanishing
constant curvature $\kappa $ and non-vanishing constant torsion $\tau $ \cite%
{bar}. \ It is known that if the curve is a straight line or plane curve, then 
$\kappa =0$ or $\tau =0,$ respectively \cite{kuh}$.$ On the other hand, a
family of curves with constant torsion but the non-constant curvature is called
anti-Salkowski curves \cite{sal}.

From the view of the differential geometry, there are different
characterizations of surfaces. Generally, these type
characterizations of surfaces are done in terms of the Gaussian curvature
and mean curvature of the surface, \cite{dede1,dede2}. On the contrary, if there are the only
characterizations of the surface with constant principal curvatures, then
these surfaces are called the isoparametric surfaces.

Relevant to this matter which is characterizations of surfaces with constant
principal curvatures, in the paper \cite{tam}, Tamura showed that complete
surfaces of constant mean curvature in $E^{3}$ on which there exist two
helical geodesics through each point are planes, spheres or circular
cylinders. Recently, Lopez et al. improved surfaces in $E^{3}$ with the
property that there exist four geodesics through each point such that every
ruled surface built with the normal lines along these geodesics is a surface
with constant mean curvature.

In this paper, we improve this characterization to the Galilean 3-space $%
\mathbf{G}_{3}$ of negative curvature. To be more precise, we investigate
the following results for surfaces in the Galilean 3-space: First, we define
surfaces in the Galilean 3-space $\mathbf{G}_{3}$ with the property that
there exist four geodesics through each point such that every surface built
with the normal lines and the binormal lines along these geodesics is a
surface, which a minimal surface, with a constant negative curvature or zero
curvature. Second, we show that the surface should be an isoparametric
surface in $\mathbf{G}_{3}$: \ A plane or a circular hyperboloid. For this
reason through our paper, we shall show the following theorems:

\begin{theorem}
Let $\ \psi $ be a connected surface in Galilean 3-space $\mathbf{G}_{3}.$
If there exist four geodesics through each point of $\ \psi $ with the
property that the normal surface constructed along these geodesics is a
minimal surface with a constant negative curvature or zero curvature, then $%
\psi $ is a plane or a circular hyperboloid.
\end{theorem}

\begin{theorem}
Let $\ \psi $ be a connected surface in Galilean 3-space $\mathbf{G}_{3}.$
If there exist four geodesics through each point of $\ \psi $ with the
property that the binormal surface constructed along these geodesics is a
minimal surface with a constant negative curvature or zero curvature, then $%
\psi $ is a plane or a circular hyperboloid.
\end{theorem}

\section{Preliminaries}

The Galilean 3-space $\mathbf{G}_{3}$ is a Cayley-Klein space equipped with
the projective metric of signature $(0,0,+,+),$ given in \cite{mol}. The
absolute figure of the Galilean space consists of an ordered triple $%
\{\omega ,f,I\}$ in which $\omega $ is the ideal (absolute) plane, $\ f$ \
is the line (absolute line) in $\omega $ and $I$ is the fixed elliptic
involution of $f$. We introduce homogeneous coordinates in $\mathbf{G}_{3}$
in such a way that the absolute plane $\omega $ is given by $x_{0}=0$, the
absolute line $f$ by $x_{0}=x_{1}=0$ and the elliptic involution by

\begin{equation}
(0:0:x_{2}:x_{3})\rightarrow (0:0:x_{3}:-x_{2}).  \label{1a}
\end{equation}%
A plane is called Euclidean if it contains $f$, otherwise it is called
isotropic or i.e., planes $x=const.$ are Euclidean, and so is the plane $%
\omega $. Other planes are isotropic. In other words, an isotropic plane
does not involve any isotropic direction.

\begin{definition}
\textrm{(\cite{Milin})} Let $x=$ $\left( x_{1},x_{2},x_{3}\right) $ and $y=$ 
$\left( y_{1},y_{2},y_{3}\right) $ be any two vectors in $\mathbf{G}_{3}.$ A
vector $a$ is called isotropic if $x_{1}=0$, otherwise it is called
non-isotropic. Then the Galilean scalar product of these vectors is given by%
\begin{equation}
\left\langle x,y\right\rangle =\left\{ 
\begin{array}{cc}
x_{1}y_{1}, & \text{if }x_{1}\neq 0\text{ or }y_{1}\neq 0 \\ 
x_{2}y_{2}+x_{3}y_{3}, & \text{if \ }x_{1}=0\text{ and }y_{1}=0%
\end{array}%
\right. .  \label{1}
\end{equation}
\end{definition}

\begin{definition}
\textrm{(\cite{Yaglom})} \textbf{\ }Let $x=$ $\left(
x_{1},x_{2},x_{3}\right) $ and \textbf{$y$}$=$ $\left(
y_{1},y_{2},y_{3}\right) $ be any two vectors in $\mathbf{G}_{3},$ the
Galilean cross product is given as 
\begin{equation}
x\mathbf{\wedge }y\mathbf{=}\left\{ \left\vert 
\begin{array}{ccc}
0 & e_{2} & e_{3} \\ 
x_{1} & x_{2} & x_{3} \\ 
y_{1} & y_{2} & y_{3}%
\end{array}%
\right\vert \right. .  \label{2}
\end{equation}
\end{definition}

\begin{definition}
\textrm{(\cite{sa})} Let $T$ be the unit tangent vector of a curve $\alpha $
on a surface $\psi $ in ${\mathbf{G}_{3}}$, and $N$ be the unit normal
vector to the surface at the point $\alpha (s)$ of $\alpha $, respectively.
Let $B=T\mathbf{\wedge }N$ be the tangential-normal. Then $\{T,N,B\}$ is an
orthonormal frame at $\alpha (s)$ in $\mathbf{G}_{3}$. The frame is called a
Galilean Darboux frame or a tangent-normal frame and it is expressed as 
\begin{equation}
\begin{aligned} T^{\prime }(s) &=k_{g}(s)B(s)+k_{n}(s)N(s), \\ N^{\prime
}(s) &=\tau _{g}(s)B(s), \\ B ^{\prime } (s)&=-\tau _{g}(s)N(s),
\end{aligned}  \label{3}
\end{equation}%
where $k_{g}$, $k_{n}$ and $\tau _{g}$ are the geodesic curvature, normal
curvature, and geodesic torsion, respectively.
\end{definition}

For the curvature $\kappa $ of $\alpha $, $\kappa ^{2}=k_{g}^{2}+k_{n}^{2}$
holds. Also, a curve $\alpha $ is a geodesic (an asymptotic curve or a line
of curvature) if and only if $k_{g}$ ( $k_{n}$ or $\tau _{g}$) vanishes,
respectively.

Suppose now that $\alpha $ is a geodesic. Then $k_{g}=0$ and $T^{\prime
}=k_{n}(s)N(s),$ which implies that, up change of orientation on $\psi $ if
necessary, $N(s)$ is the normal vector to $\alpha .$ So, from the Galilean
Darboux frame, we can write $k_{n}=\kappa $ and $\tau _{g}=\tau $. If
replacing these in Equation\ 
%TCIMACRO{\TeXButton{3}{\eqref{3}}}%
%BeginExpansion
\eqref{3}%
%EndExpansion
, then the formulae becomes 
\begin{eqnarray}
T^{\prime }(s) &=&\kappa (s)N(s),  \label{3a} \\
N^{\prime }(s) &=&\tau (s)B(s),  \notag \\
B^{\prime }(s) &=&-\tau (s)N(s).  \notag
\end{eqnarray}

Let the equation of a surface $\psi =\psi (s,t)$ in $\mathbf{G}_{3}$ is
given by%
\begin{equation*}
\psi (s,t)=\left( x(s,t),y(s,t),z(s,t)\right) .
\end{equation*}%
Then the unit isotropic normal vector field $\eta $ on $\psi (s,t)$ is given
by 
\begin{equation}
\eta =\frac{\psi _{,s}\mathbf{\wedge }\psi _{,t}}{\left\Vert \psi _{,s}%
\mathbf{\wedge }\psi _{,t}\right\Vert },  \label{4}
\end{equation}%
where the partial differentiations with respect to $s$ and $t$, that is, $%
\psi _{,s}=\dfrac{\partial \psi (s,t)}{\partial s}$ and $\psi _{,t}=\dfrac{%
\partial \psi (s,t)}{\partial t}.$

From 
%TCIMACRO{\TeXButton{1a}{\eqref{1a}} }%
%BeginExpansion
\eqref{1a}
%EndExpansion
and $w=\left\Vert \psi _{,s}\mathbf{\wedge }\psi _{,t}\right\Vert $, we get
the isotropic unit vector $\delta $ in the tangent plane of the surface as%
\begin{equation}
\delta =\frac{x_{,2}\psi _{,s}-x_{,1}\psi _{,t}}{w},  \label{4a}
\end{equation}%
where $x_{,1}=\dfrac{\partial x(s,t)}{\partial s}$, $x_{,2}=\dfrac{\partial
x(s,t)}{\partial t}$ and $\left\langle \eta ,\delta \right\rangle =0$, $%
\delta ^{2}=1.$

Let us define 
\begin{equation*}
\begin{aligned} & g_{1}=x _{,1},\text{ }g_{2}=x _{,2}, \text{ } g_{ij}=
g_{i}g_{j},\\ & g^{1}= \frac {x_{,2}}{w}, \text{ } g^2 =-\frac {x_{,1}}{w},
\text{ } g^{ij}= g^i g^j ~~ (i,j=1,2), \\ & h_{11}= \langle\widetilde{\psi}
_{,1}, \text{ }\widetilde{\psi} _{,1}\rangle, \text{ } h_{12}= \langle
\widetilde{\psi} _{,1} ,\widetilde{\psi}_{,2}\rangle, \text{ } h_{22}=
\langle \widetilde{\psi} _{,2} ,\widetilde{\psi}_{,2}\rangle, \end{aligned}
\end{equation*}%
where $\widetilde{\psi} _{,1}$ and $\widetilde{\psi} _{,2}$ are the
projections of vectors {$\psi $}$_{,1}$ and {$\psi $}$_{,2}$ onto the $yz$%
-plane, respectively. Then, the corresponding matrix of the first
fundamental form $ds^{2}$ of the surface $\psi (s,t)$ is given by (cf. \cite%
{SD}) 
\begin{equation}
ds^{2}=%
\begin{pmatrix}
ds_{1}^{2} & 0 \\ 
0 & ds_{2}^{2}%
\end{pmatrix}%
,  \label{a4}
\end{equation}%
where $ds_{1}^{2}=(g_{1}ds+g_{2}dt)^{2}$ and $%
ds_{2}^{2}=h_{11}ds^{2}+2h_{12}dsdt+h_{22}dt^{2}$. In such case, we denote
the coefficients of $ds^{2}$ by $g_{ij}^{\ast }$.

%Then, the first fundamental form of $\Psi (s,\vartheta )$ is given by
%$$
%ds^2 = (g_1 du_1 + g_2 du_2 )^2 + \epsilon ( h_{11}
%du_1 ^2 + 2  h_{12}du_1 du_2 + h_{22} du_2^2 ),
%$$
%where
%$$
%\epsilon =
%\begin{cases}
%0, \quad  &  \text{ if direction}~ ds : d \vartheta ~ \text {is nonisotropic},\\
%1, \quad   & \text { if direction}~   ds : d\vartheta ~ \text {is isotropic.}
%\end{cases}
%$$

\noindent On the other hand, the function $w$ can be represented in terms of 
$g_i$ and $h_{ij}$ as follows: 
\begin{equation*}
w^2 = g_1^2 h_{22} - 2 g_1 g_2 h_{12} + g_2^2 h_{11}.
\end{equation*}

The Gaussian curvature and the mean curvature of a surface is defined by
means of the coefficients of the second fundamental form $L_{ij}$, which are
the normal components of $\psi _{,i,j}$ $(i,j=1,2)$. That is, 
\begin{equation*}
\psi _{,i,j}=\sum_{k=1}^{2}\Gamma _{ij}^{k}\psi _{,k}+L_{ij}\eta ,
\end{equation*}%
where $\Gamma _{ij}^{k}$ is the Christoffel symbols of the surface and $%
L_{ij}$ are given by 
\begin{equation}
L_{ij}=\frac{1}{g_{1}}\langle g_{1}\psi _{,i,j}-g_{i,j}\psi _{,1},\eta
\rangle =\frac{1}{g_{2}}\langle g_{2}\psi _{,i,j}-g_{i,j}\psi _{,2},\eta
\rangle .  \label{7a}
\end{equation}%
From this, the Gaussian curvature $K$ and the mean curvature $H$ of the
surface are expressed as \cite{ro} 
\begin{equation}
\begin{aligned} K & = \frac {L_{11}L_{22}-L_{12}^2}{ w^2},\\ H &= \frac {
g_2^2 L_{11}- 2 g_1 g_2 L_{12} + g_1^2 L_{22}}{2 w^2} . \end{aligned}
\label{7b}
\end{equation}

\section{ Proof of Theorem 1.1}

In this section, the below definition, proposition and lemma to prove of
Theorem 1.1 are given in different steps.

Now we will start to give the definition of the normal surface as follows:

\begin{definition}
Let $\psi $ be a surface in $\mathbf{G}_{3}$. The normal surface $\varphi $
through $\alpha $ is the surface whose the base curve is $\alpha $ and the
ruling are the straight-lines orthogonal to $\psi $ through $\alpha $.
\end{definition}

The normal surface $\varphi $ along $\alpha $ is a regular surface at least
around $\alpha .$ Then $\varphi $ is parametrized by 
\begin{equation}
\varphi (s,t)=\alpha (s)+tN(s),  \label{8}
\end{equation}%
where $s\in I$ and $t\in R.$ Then we obtain 
\begin{equation}
\varphi _{,s}=\alpha ^{\prime }(s)+t(dN)_{\alpha (s)}(\alpha ^{\prime }(s))
\label{8a}
\end{equation}%
and 
\begin{equation}
\varphi _{,t}=N(s).  \label{8b}
\end{equation}

Considering $s_{0}\in I,$ we get \ $\left\Vert \varphi _{,s}\mathbf{\wedge }%
\varphi _{,t}\right\Vert (s_{0},0)=\left\Vert \alpha ^{\prime }(s_{0})%
\mathbf{\wedge }N(\alpha ^{\prime }(s_{0}))\right\Vert \neq 0.$ Thus, from
the inverse function theorem, $\varphi (s,t)$ is an immersion.

Our first step proving Theorem 1.1 is to show the Gaussian curvature of the
normal surface build up along the geodesic is a negative constant curvature or
zero curvature and should be minimal.

\begin{proposition}
Let $\psi $ be a connected surface in $\mathbf{G}_{3}$. If the normal
surface $\varphi $ constructed along a geodesic of $\psi $ is a surface with
constant Gaussian curvature, then $\varphi $ is either a constant negative
curvature surface or flat surface and $\varphi $ should be a minimal surface.
\end{proposition}

\begin{proof}
Firstly, suppose that $\alpha $ is not a straightline. Then its curvature is
defined as well as $T^{\prime }(s)\neq 0.$ Considering equations 
%TCIMACRO{\TeXButton{8a}{\eqref{8a}} }%
%BeginExpansion
\eqref{8a}
%EndExpansion
and 
%TCIMACRO{\TeXButton{8b}{\eqref{8b}}}%
%BeginExpansion
\eqref{8b},%
%EndExpansion
\ 

\begin{equation}
\varphi _{,s}\mathbf{\wedge }\varphi _{,t}=(0,0,1),  \label{9c}
\end{equation}
then we have 
\begin{equation}
\left\Vert \varphi _{,s}\mathbf{\wedge }\varphi ,_{t}\right\Vert =1.
\label{9d}
\end{equation}%
Using 
%TCIMACRO{\TeXButton{4}{\eqref{4}}}%
%BeginExpansion
\eqref{4}%
%EndExpansion
, the unit isotropic normal vector $\eta _{\varphi }$ of $\varphi (s,t)$ is
found as%
\begin{equation}
\eta _{\varphi }=(0,0,1).  \label{10}
\end{equation}%
On the other hand, from equations 
%TCIMACRO{\TeXButton{10}{\eqref{10}} }%
%BeginExpansion
\eqref{10}
%EndExpansion
and 
%TCIMACRO{\TeXButton{1a}{\eqref{1a}}}%
%BeginExpansion
\eqref{1a}%
%EndExpansion
, it is easy to show that 
\begin{equation*}
\delta _{\varphi }=(0,-1,0).
\end{equation*}%
Since $\eta _{\varphi }$ is the isotropic vector, using the Galilean frame,
we can obtain \ $g_{\varphi 1}=1,$ $g_{\varphi 2}=0.$

Considering the projection of $\varphi _{,s}$ and $\varphi _{,t}$ onto the
Euclidean $yz-$plane, we obtain 
\begin{equation}
h_{\varphi 22}=1.  \label{10b}
\end{equation}%
Using 
%TCIMACRO{\TeXButton{8b}{\eqref{8b}}}%
%BeginExpansion
\eqref{8b}%
%EndExpansion
, the coefficients of the first fundamental form of the surface in Galilean
space are obtained as%
\begin{equation*}
g_{\varphi 11}^{\ast }=1,\quad g_{\varphi 12}^{\ast }=0,\quad g_{\varphi
22}^{\ast }=1.
\end{equation*}

\noindent To calculate the second fundamental form of $\varphi (s,t)$, we
have to calculate the following: 
\begin{eqnarray}
\varphi _{,s,s} &=&\left( \kappa -t\tau ^{2}\right) N+t\tau ^{\prime }B
\label{10c} \\
\varphi _{,t,s} &=&\tau B,  \notag \\
\varphi _{,t,t} &=&0.  \notag
\end{eqnarray}%
From equations 
%TCIMACRO{\TeXButton{10c}{\eqref{10c}} }%
%BeginExpansion
\eqref{10c}
%EndExpansion
and 
%TCIMACRO{\TeXButton{7a}{\eqref{7a}}}%
%BeginExpansion
\eqref{7a}%
%EndExpansion
, the coefficients of the second fundamental form are found as%
\begin{equation}
L_{\varphi 11}=t\tau ^{\prime },\text{ }L_{\varphi 12}=\tau ,\text{ }%
L_{\varphi 22}=0.  \label{10d}
\end{equation}%
Thus, $K_{\varphi }$ and $H_{\varphi }$ are calculated as 
\begin{equation}
K_{\varphi }=-\tau ^{2}<0,  \label{11}
\end{equation}%
\begin{equation}
H_{\varphi }=0.  \label{12}
\end{equation}%
Secondly, assume that $\alpha $ is a straightline. Then the similar
calculations as first case can be done as follows:

\begin{eqnarray}
\varphi _{,s} &=&T(s)+tN^{\prime }(s),\text{ }\varphi _{,t}=N(s),  \label{13}
\\
\varphi _{,s,s} &=&tN^{\prime \prime }(s),\text{ }\varphi _{,t,s}=N^{\prime
}(s)\text{ and }\varphi _{,t,t}=0,  \notag
\end{eqnarray}%
where $\alpha ^{\prime }(s)=T(s)$ is the tangent vector to $\alpha .$

Using 
%TCIMACRO{\TeXButton{4}{\eqref{4}}}%
%BeginExpansion
\eqref{4}%
%EndExpansion
, the unit isotropic normal vector $\eta _{\varphi }$ of $\varphi (s,t)$ is
obtained as%
\begin{equation}
\eta _{\varphi }=\frac{T(s)\mathbf{\wedge }N(s)+tN^{\prime }(s)\mathbf{%
\wedge }N(s)}{w_{\varphi }},  \label{14}
\end{equation}%
where $w_{\varphi }=\left\Vert T(s)\mathbf{\wedge }N(s)+tN^{\prime }(s)%
\mathbf{\wedge }N(s)\right\Vert .$

From equations 
%TCIMACRO{\TeXButton{10c}{\eqref{10c}} }%
%BeginExpansion
\eqref{10c}
%EndExpansion
and 
%TCIMACRO{\TeXButton{7a}{\eqref{7a}}}%
%BeginExpansion
\eqref{7a}%
%EndExpansion
, the coefficients of the second fundamental form are given as

\begin{eqnarray}
L_{\varphi 11} &=&\frac{1}{w_{\varphi }}\langle tN^{\prime \prime }(s),T(s)%
\mathbf{\wedge }N(s)+tN^{\prime }(s)\mathbf{\wedge }N(s)\rangle  \label{15}
\\
L_{\varphi 11} &=&\frac{t}{w_{\varphi }}\langle N^{\prime \prime }(s),T(s)%
\mathbf{\wedge }N(s)\rangle +\frac{t^{2}}{w}\langle N^{\prime \prime
}(s),N^{\prime }(s)\mathbf{\wedge }N(s)\rangle ,  \notag \\
L_{\varphi 12} &=&\frac{1}{w_{\varphi }}\langle N^{\prime }(s),T(s)\mathbf{%
\wedge }N(s)+tN^{\prime }(s)\mathbf{\wedge }N(s)\rangle  \notag \\
L_{\varphi 12} &=&\frac{1}{w_{\varphi }}\langle N^{\prime }(s),T(s)\mathbf{%
\wedge }N(s)\rangle ,  \notag \\
L_{\varphi 22} &=&0.  \notag
\end{eqnarray}%
Thus, the Gaussian curvature $K_{\varphi }$ satisfies 
\begin{equation}
K_{\varphi }w_{\varphi }^{\frac{3}{2}}=\langle N^{\prime }(s),T(s)\mathbf{%
\wedge }N(s)\rangle ^{2}.  \label{16}
\end{equation}%
Squaring both sides equation 
%TCIMACRO{\TeXButton{16}{\eqref{16}} }%
%BeginExpansion
\eqref{16}
%EndExpansion
and writing as polynomial equation, we get a polynomial on $t$ of degree
six, this means that%
\begin{equation*}
\sum_{n=1}^{6}P_{n}(s)t^{n}=0.
\end{equation*}%
Particularly, $P_{n}(s)=0$ for $0\leq n\leq 6.$ Then we can obtain $\
P_{0}(s)=K_{\varphi }^{2}=0,$ which implies that $K_{\varphi }=0.$

Furthermore we can easily show that 
\begin{equation*}
H_{\varphi }=0,
\end{equation*}%
which is completed the proof.
\end{proof}

\begin{lemma}
Let $\psi $ be a connected surface in $\mathbf{G}_{3}$. If the normal
surface $\varphi $, which is a minimal surface, constructed along a geodesic 
$\alpha $ is a constant negative curvature surface or a flat surface, then $%
\alpha $ is either

$i)$ an anti Salkowski curve,

$ii)$ a planar curve,

$iii)$ or a line segment and the last case $\alpha $ is a line of curvature
of $\psi .$
\end{lemma}

\begin{proof}
Let $\alpha $ be not a line segment. Then $\kappa >0$. Since $\varphi $ is a
surface with a negative constant curvature, from equation 
%TCIMACRO{\TeXButton{11}{\eqref{11}}}%
%BeginExpansion
\eqref{11}%
%EndExpansion
, we have 
\begin{equation*}
\tau (s)=const.,
\end{equation*}%
for all $s$, as a result, we have constant torsion but non-constant
curvature. This means that $\alpha $ is an anti Salkowski curve \textrm{\cite%
{sal}}. Or, we have 
\begin{equation*}
\tau (s)=0,
\end{equation*}%
this implies $\alpha $ is a planar curve.

In the last case, if $\alpha $ is a line segment, from equation 
%TCIMACRO{\TeXButton{16}{\eqref{16}}}%
%BeginExpansion
\eqref{16}%
%EndExpansion
, we get 
\begin{equation}
\langle N^{\prime }(s),T(s)\mathbf{\wedge }N(s)\rangle =0.  \label{17}
\end{equation}%
Taking account of the above equations, we can write 
\begin{equation*}
T(s)=a(s)N(s)+b(s)N^{\prime }(s).
\end{equation*}%
Since $\left\langle T(s),N(s)\right\rangle =\left\langle N^{\prime
}(s),N(s)\right\rangle =0,$ we find 
\begin{equation*}
T(s)=b(s)N^{\prime }(s),
\end{equation*}%
this means that 
\begin{equation*}
N^{\prime }(s)=\lambda T(s),
\end{equation*}%
from this $\alpha $ is also a line of curvature of the surface.
\end{proof}

\bigskip

Therefore, Lemma 3.3 means that, under the same hypothesis of Theorem 1.1,
there exist four geodesics through each point $p\in \psi $ which are the
next three types:

$i)$ An anti Salkowski curve (Type 1),

$ii)$ A planar curve (Type 2),

$iii)$ A line segment, which is a line of curvature (Type 3)$.$

\begin{theorem}
(\cite{tam2}) A connected surface in $R^{3}$ with the property that there
exist two proper helical geodesics through each point of the surface is an
open of a right circular cylinder.
\end{theorem}

Now our aim is to give the proof of Theorem 1.1. Considering Theorem 3.4 and
Proposition 3.2, we can give the following claim:

\begin{claim}
Let $p\in \psi $ a non-umbilic point. In a neighborhood $A_{\varphi }$ of $p$%
, there are two proper helical geodesics which is a curve that is both a
proper circular helix and a geodesic on $\psi ,$ through any point of $%
A_{\varphi }$.
\end{claim}

By Lemma 3.3, if $A_{\varphi }\subset $ $\psi $ is an open set around $p$
formed by non-umbilic points, then there exist four tangent directions at $%
q\in A_{\varphi }$ such that the corresponding geodesic refers to one of the
above three Types 1, 2, 3. Because the point is not umbilic there are at
most two geodesics of Type 2 or Type 3. As there are four geodesics of Types
either 1, 2, or 3, we have two geodesics which are an anti salkowski curve,
i.e. of Type 1. In particular, If we get $\kappa =const.,$ then the anti
salkowski curve turns out to be a circular helix. This proves the claim.

Let us denote $\psi _{1}$ is the set of umbilic points of $\psi $. This set
is closed on $\psi .$

$i)$ If $\psi -\psi _{1}\neq \emptyset $, then $\psi -\psi _{1}$ is
contained in circular hyperboloid . In particular, we can write $K_{\varphi
}=-\dfrac{1}{r^{2}}<0$ and $H_{\varphi }=0$ on $\psi -\psi _{1}$. \ Thus we
can define closed set in $\psi $ such that $\psi _{2}=\left\{ p\in \psi
:K_{\varphi }(p)=-\dfrac{1}{r^{2}}<0,\text{ }H_{\varphi }(p)=0\right\} $.
From connectedness, we get proved that $\psi -\psi _{1}\subset $ $\psi _{2}$%
. Since $\psi _{2}\cap $ $\psi _{1}=\emptyset $, \ we easily say that $\psi
_{2}\subset \psi -\psi _{1}$ , which proves that $\psi _{2}=\psi -\psi _{1}$%
. As $\psi _{2}$ is both an open and closed set of $\psi $, $\psi _{2}$ $=$ $%
\psi $ by connectedness, proving that $\psi $ is an open set of a circular
hyperboloid.

$ii)$ If $\psi -\psi _{1}=\emptyset $, then $\psi $ is an umbilic surface.
Then $\psi $ is an open of a plane since we have a flat and a minimal
surface.

Then we finish the proof of Theorem 1.1.

\section{ Proof of Theorem 1.2}

In this section, we give the proof of Theorem 1.2 in different steps as the
proof of Theorem 1.1 .

Now, our starting point is to give the definition of the binormal surface as
follows:

\begin{definition}
Let $\psi $ be a surface in $\mathbf{G}_{3}$. The binormal surface $\phi $
through $\alpha $ is the surface whose the base curve is $\alpha $ and the
ruling are the straightlines orthogonal to $\phi $ through $\alpha $.
\end{definition}

The binormal surface $\phi $ along $\alpha $ is a regular surface at least
around $\alpha .$ Then $\phi $ is specifed by 
\begin{equation}
\phi (s,t)=\alpha (s)+tB(s),  \label{18}
\end{equation}%
where $s\in I$ and $t\in R.$ Then we get%
\begin{equation}
\phi _{,s}=\alpha ^{\prime }(s)+t(dB)_{\alpha (s)}(\alpha ^{\prime }(s))
\label{18a}
\end{equation}%
and 
\begin{equation}
\phi _{,t}=B(s).  \label{18b}
\end{equation}

If we consider $s_{0}\in I,$ then we get \ $\left\Vert \phi _{,s}\mathbf{%
\wedge }\phi _{,t}\right\Vert (s_{0},0)=\left\Vert \alpha ^{\prime }(s_{0})%
\mathbf{\wedge }B(\alpha ^{\prime }(s_{0}))\right\Vert \neq 0.$ Thus, from
the inverse function theorem, $\phi (s,t)$ is an immersion.

Our first step is to prove Theorem 1.2 is to get the Gaussian curvature of
the binormal surface constructed along a geodesic is a negative constant
curvature or zero curvature and should be minimal.

\begin{proposition}
Let $\psi $ be a connected surface in $\mathbf{G}_{3}$. If the binormal
surface $\phi $ constructed along a geodesic of $\psi $ is a surface with
constant Gaussian curvature, then $\phi $ is either a constant negative
curvature surface or flat surface and $\phi $ should be a minimal surface.
\end{proposition}

\begin{proof}
Assume first that $\alpha $ is not a straightline. Then its curvature is
defined. If we consider equations 
%TCIMACRO{\TeXButton{18a}{\eqref{18a}} }%
%BeginExpansion
\eqref{18a}
%EndExpansion
and 
%TCIMACRO{\TeXButton{18b}{\eqref{18b}}}%
%BeginExpansion
\eqref{18b}%
%EndExpansion
,

\begin{equation}
\phi _{,s}\mathbf{\wedge }\phi _{,t}=(0,-1,0),  \label{19c}
\end{equation}%
then we have
\begin{equation}
\left\Vert \phi _{,s}\mathbf{\wedge }\phi ,_{t}\right\Vert =1.  \label{19d}
\end{equation}%
From 
%TCIMACRO{\TeXButton{4}{\eqref{4}}}%
%BeginExpansion
\eqref{4}%
%EndExpansion
, the unit isotropic normal vector $\eta _{\phi }$ of $\phi (s,t)$ is found
as%
\begin{equation}
\eta _{\phi }=(0,-1,0).  \label{20}
\end{equation}%
On the other hand, from equations 
%TCIMACRO{\TeXButton{20}{\eqref{20}} }%
%BeginExpansion
\eqref{20}
%EndExpansion
and 
%TCIMACRO{\TeXButton{1a}{\eqref{1a}}}%
%BeginExpansion
\eqref{1a}%
%EndExpansion
, it is easy to calculate that 
\begin{equation*}
\delta _{\phi }=(0,0,-1).
\end{equation*}%
Since $\eta _{\phi }$ is the isotropic vector, using the Galilean frame, we
can get \ $g_{\phi 1}=1,$ $g_{\phi 2}=0.$

Considering the projection of $\phi _{,s}$ and $\phi _{,t}$ onto the
Euclidean $yz-$plane, we get 
\begin{equation}
h_{\phi 22}=1.  \label{20b}
\end{equation}%
The coefficients of the first fundamental form of the surface in Galilean
space are found as%
\begin{equation*}
g_{\phi 11}^{\ast }=1,\quad g_{\phi 12}^{\ast }=0,\quad g_{\phi 22}^{\ast
}=1.
\end{equation*}

\noindent To calculate the second fundamental form of $\phi (s,t)$, we have
to compute the following: 
\begin{eqnarray}
\phi _{,s,s} &=&\left( \kappa -t\tau ^{\prime }\right) N-t\tau ^{2}B
\label{20c} \\
\phi _{,t,s} &=&-\tau N,  \notag \\
\phi _{,t,t} &=&0.  \notag
\end{eqnarray}%
From equations 
%TCIMACRO{\TeXButton{20c}{\eqref{20c}} }%
%BeginExpansion
\eqref{20c}
%EndExpansion
and 
%TCIMACRO{\TeXButton{7a}{\eqref{7a}}}%
%BeginExpansion
\eqref{7a}%
%EndExpansion
, the coefficients of the second fundamental form are found as%
\begin{equation}
L_{\phi 11}=-\kappa +t\tau ^{\prime },\text{ }L_{\phi 12}=\tau ,\text{ }%
L_{\phi 22}=0.  \label{20d}
\end{equation}%
Thus, $K_{\phi }$ and $H_{\phi }$ are computed as 
\begin{equation}
K_{\phi }=-\tau ^{2}<0,  \label{21}
\end{equation}%
\begin{equation}
H_{\phi }=0.  \label{22}
\end{equation}%
Secondly, suppose that $\alpha $ is a straightline. Then the similar
calculations as first case can be done as follows:

\begin{eqnarray}
\phi _{,s} &=&T(s)+tB^{\prime }(s),\text{ }\phi _{,t}=B(s),  \label{23} \\
\phi _{,s,s} &=&tB^{\prime \prime }(s),\text{ }\phi _{,t,s}=B^{\prime }(s)%
\text{ and }\phi _{,t,t}=0,  \notag
\end{eqnarray}%
where $\alpha ^{\prime }(s)=T(s)$ is the tangent vector to $\alpha .$

Using 
%TCIMACRO{\TeXButton{4}{\eqref{4}}}%
%BeginExpansion
\eqref{4}%
%EndExpansion
, the unit isotropic normal vector $\eta _{\phi }$ of $\phi (s,t)$ is
obtained as%
\begin{equation}
\eta _{\phi }=\frac{T(s)\mathbf{\wedge }B(s)+tB^{\prime }(s)\mathbf{\wedge }%
B(s)}{w_{\phi }},  \label{24}
\end{equation}%
where $w_{\phi }=\left\Vert T(s)\mathbf{\wedge }B(s)+tB^{\prime }(s)\mathbf{%
\wedge }B(s)\right\Vert .$

From equations 
%TCIMACRO{\TeXButton{20c}{\eqref{20c}} }%
%BeginExpansion
\eqref{20c}
%EndExpansion
and 
%TCIMACRO{\TeXButton{7a}{\eqref{7a}}}%
%BeginExpansion
\eqref{7a}%
%EndExpansion
, the coefficients of the second fundamental form are given as

\begin{eqnarray}
L_{\phi 11} &=&\frac{1}{w_{\phi }}\langle tB^{\prime \prime }(s),T(s)\mathbf{%
\wedge }B(s)+tB^{\prime }(s)\mathbf{\wedge }B(s)\rangle  \label{25} \\
L_{\phi 11} &=&\frac{t}{w_{\phi }}\langle B^{\prime \prime }(s),T(s)\mathbf{%
\wedge }B(s)\rangle +\frac{t^{2}}{w}\langle B^{\prime \prime }(s),B^{\prime
}(s)\mathbf{\wedge }B(s)\rangle ,  \notag \\
L_{\phi 12} &=&\frac{1}{w_{\phi }}\langle B^{\prime }(s),T(s)\mathbf{\wedge }%
B(s)+tB^{\prime }(s)\mathbf{\wedge }B(s)\rangle  \notag \\
L_{\phi 12} &=&\frac{1}{w_{\phi }}\langle B^{\prime }(s),T(s)\mathbf{\wedge }%
B(s)\rangle ,  \notag \\
L_{\phi 22} &=&0.  \notag
\end{eqnarray}%
Thus, the Gaussian curvature $K_{\phi }$ satisfies 
\begin{equation}
K_{\phi }w_{\phi }^{\frac{3}{2}}=\langle B^{\prime }(s),T(s)\mathbf{\wedge }%
B(s)\rangle ^{2}.  \label{26}
\end{equation}%
Squaring both sides equation 
%TCIMACRO{\TeXButton{26}{\eqref{26}} }%
%BeginExpansion
\eqref{26}
%EndExpansion
and writing as polynomial equation, we obtain a polynomial on $t$ of degree
six, this means that%
\begin{equation*}
\sum_{n=1}^{6}Q_{n}(s)t^{n}=0.
\end{equation*}%
Particularly, $Q_{n}(s)=0$ for $0\leq n\leq 6.$ Then we can obtain $\
Q_{0}(s)=K_{\phi }^{2}=0,$ which implies that $K_{\phi }=0.$

Moreover we can easily show that 
\begin{equation*}
H_{\phi }=0.
\end{equation*}%
Hence this completes the proof.
\end{proof}

\begin{lemma}
Let $\psi $ be a connected surface in $G_{3}$. If the binormal surface $\phi 
$, which is a minimal surface, constructed along a geodesic $\alpha $ is a
constant negative curvature surface or a flat surface, then $\alpha $ is
either

$i)$ an anti Salkowski curve,

$ii)$ a planar curve,

$iii)$ or a line segment$.$
\end{lemma}

\begin{proof}
This proof can be done in a similar way to the proof of Lemma 3.3.
\end{proof}

Therefore, Lemma 4.3 implies that, under the same hypothesis of Theorem 1.2,
there exist four geodesics through each point $p\in \psi $ which are the
next three types:

$i)$ An anti Salkowski curve (Type 1),

$ii)$ A planar curve (Type 2),

$iii)$ A line segment (Type 3)$.$

\begin{claim}
Let $p\in \psi $ a non-umbilic point. In a neighborhood $A_{\phi }$ of $p$,
there are two proper helical geodesics which is a curve that is both a
proper circular helix and a geodesic on $\psi ,$ through any point of $%
A_{\phi }$.
\end{claim}

Then we finish the proof of Theorem 1.2.

\end{document}